\documentclass[11pt]{amsart}
\usepackage[top=1.2in,bottom=1.2in,left=1.2in,right=1.2in]{geometry}
\usepackage[inline]{enumitem}

\usepackage{graphicx}
\usepackage{tikz}
\usetikzlibrary{positioning}
\usetikzlibrary{arrows}
\usetikzlibrary{matrix,arrows,decorations.pathmorphing}

\usepackage{hyperref}
\hypersetup{
    colorlinks,
    citecolor=black,
    filecolor=black,
    linkcolor=black,
    urlcolor=black
}

\usepackage{subcaption}

\usepackage{multicol}

\usepackage{amsthm}

\usepackage{mathtools}

\usepackage{graphicx}
\usepackage{amssymb}
\usepackage{epstopdf}
\usepackage{hyperref}
\usepackage{xypic}

\newcommand\Z{\mathbb{Z}}

\newcommand{\CRS}{\Gamma}

\newcommand{\D}{D}
\newcommand{\B}{B}

\DeclareMathOperator{\Aut}{Aut}
\DeclareMathOperator{\Mod}{Mod}
\DeclareMathOperator{\Stab}{Stab}

\theoremstyle{plain}

\newtheorem{theorem}{Theorem}[section]

\newtheorem{lemma}[theorem]{Lemma}

\newcommand{\p}[1]{\bigskip \noindent \emph{#1}.}

\usepackage{pinlabel}

\title{Homomorphisms of commutator subgroups of braid groups}

\author{Kevin Kordek}
\author{Dan Margalit}

\address{Kevin Kordek \\ School of Mathematics\\ Georgia Institute of Technology \\ 686 Cherry St. \\ Atlanta, GA 30332}
\email{kevin.kordek@math.gatech.edu}

\address{Dan Margalit \\ School of Mathematics\\ Georgia Institute of Technology \\ 686 Cherry St. \\ Atlanta, GA 30332}
\email{margalit@math.gatech.edu}

\thanks{Both authors were supported by the National Science Foundation under Grant No. DMS - 1057874.}

\date{}

\begin{document}

\maketitle

\vspace{-2em}

\begin{abstract}
We give a complete classification of homomorphisms from the commutator subgroup of the braid group on $n$ strands to the braid group on $n$ strands when $n$ is at least~7.  In particular, we show that each nontrivial homomorphism extends to an automorphism of the braid group on $n$ strands.  This answers four questions of Vladimir Lin.  Our main new tool is the theory of totally symmetric sets.
\end{abstract}

\section{Introduction}

 Let $\B_n$ denote the braid group on $n$ strands and let $\B_n'$ denote its commutator subgroup.  We say that two homomorphisms $\rho_1 :  \B_n' \to \B_n$ and $\rho_2 :  \B_n' \to \B_n$ are \emph{equivalent} if there is an automorphism $\alpha$ of $\B_n$ such that $\alpha\circ \rho_1 = \rho_2$.  The following is our main result.
 
\begin{theorem}\label{thm:main}
Let $n\geq 7$, and let $\rho : \B_n'\rightarrow \B_n$ be a nontrivial homomorphism.  Then $\rho$ is equivalent to the inclusion map.
\end{theorem}

In his 1996 preprint, Vladimir Lin asks the following four questions about endomorphisms of $\B_n'$ \cite[0.9.2(b)--0.9.2(e)]{linpreprint}:
\begin{itemize}
\item Is every nontrivial endomorphism of $\B_n'$ injective?
\item Is every nontrivial endomorphism of $\B_n'$ equal to an automorphism of $\B_n'$?
\item Does every nontrivial endomorphism of $\B_n'$ extend to an endomorphism of $\B_n$? 
\item Does every nontrivial endomorphism of $\B_n'$ extend to an automorphism of $\B_n$?
\end{itemize}
The second and fourth questions also appear in the online problem list ``Open problems in combinatorial and geometric group theory'' \cite[Problems B5(b) and B7(b)]{shpilrain} and in the published version of the problem list \cite[Problems B6(b) and B8(b)]{baumslag}.  Theorem~\ref{thm:main} answers all four of these questions for $n \geq 7$.  Indeed, Theorem~\ref{thm:main} implies that every nontrivial endomorphism of $\B_n'$ extends to an automorphism of $\B_n$, immediately answering the fourth question.  The third question is hence answered because automorphisms are endomorphisms.  And since an automorphism of any group restricts to an automorphism of any characteristic subgroup, this answers the first two questions as well.  

\p{Prior results} In 2017, Orevkov \cite{orevkov} showed for $n \geq 4$ that $\Aut(\B_n') \cong \Aut(\B_n)$.  Another proof of Orevkov's result for $n \geq 7$ was given by McLeay \cite{mcleay}.  Theorem~\ref{thm:main} gives another proof of Orevkov's result for $n \geq 7$.  

In his 2004 paper, Lin \cite[Theorem A]{linbp} proved there are no nontrivial homomorphisms from $\B_n'$ to $\B_m$ when $n \geq 5$ and $m < n$.  Theorem~\ref{thm:main} also implies Lin's result for $n \geq 7$.  

In 1981, Dyer--Grossman \cite{dg} proved for $n \geq 3$ that $\Aut(\B_n) \cong \B_n/Z(\B_n) \rtimes \Z/2$, solving a problem of Artin \cite{artin} from 1947.  Then in 2016 Castel \cite{castel} classified for $n \geq 6$ all homomorphisms $\B_n \to \B_{n+1}$ (recently Chen and the authors \cite{chenkordekmargalit} generalized Castel's result by classifying all homomorphisms $\B_n \to \B_{2n}$ for $n \geq 5$).  Castel's theorem implies the theorem of Dyer--Grossman.  In Section~\ref{sec:castel} we explain how our Theorem~\ref{thm:main} gives a new proof of Castel's classification of endomorphisms of $\B_n$  for $n \geq 7$.  In particular, our work gives a new proof of the Dyer--Grossman result for $n \geq 7$.  

\p{New tool: totally symmetric sets} The main new tool we use to prove Theorem~\ref{thm:main} is the notion of a totally symmetric set, which we define in Section~\ref{sec:tss}.  Briefly, a totally symmetric set in a group $G$ is a subset $X$ of commuting elements with the property that each permutation of $X$ can be achieved by a single conjugation in $G$.  Recently, totally symmetric sets have also been used by Chudnovsky, Li, Partin, and the first author to give a lower bound for the cardinality of a finite non-abelian quotient of the braid group \cite{reu}.  

\p{Spaces of polynomials} Theorem~\ref{thm:main} has implications for spaces of polynomials.   Let $\textrm{Poly}_n$ denote the space of monic, square-free polynomials of degree $n$.  This is the same as the space of unordered configurations of $n$ distinct points in the plane (the $n$ points are the roots).  The fundamental group $\pi_1(\textrm{Poly}_n)$ is isomorphic to $\B_n$.

Similarly, let $\textrm{SPoly}_n$ denote the space of monic, square-free polynomials of degree $n$ and discriminant 1. The discriminant gives a map $\textrm{Poly}_n \to \mathbb{C} \setminus \{0\}$; this map is a fiber bundle with fiber $\textrm{SPoly}_n$.  Since $\mathbb{C} \setminus \{0\}$ is a $K(G,1)$ space it follows that $\pi_1(\textrm{SPoly}_n)$ embeds into $\pi_1(\textrm{Poly}_n)$ and the isomorphism from $\pi_1(\textrm{Poly}_n)$ to $\B_n$ induces an isomorphism from $\pi_1(\textrm{SPoly}_n)$ to $\B_n'$.  Because of these identifications, Theorem~\ref{thm:main} gives constraints on maps from $\textrm{SPoly}_n$ to $\textrm{Poly}_n$.

\p{Outline of the paper} In Section~\ref{sec:tss}, we introduce totally symmetric sets.  We also prove the following fundamental lemma: the image of a totally symmetric set under a homomorphism is either a totally symmetric set of the same cardinality or a singleton (Lemma~\ref{lem:tss}).  The section culminates with a classification of certain totally symmetric subsets of $\B_n$ (Lemma~\ref{lem:bnclass}).  In Section~\ref{sec:proof} we prove Theorem~\ref{thm:main} using the classification of totally symmetric sets and the fundamental lemma.  Finally, in Section~\ref{sec:castel}, we apply Theorem~\ref{thm:main} to prove the aforementioned special case of Castel's theorem.  

\p{Acknolwedgments} We are grateful to Lei Chen and Justin Lanier for helpful discussions about totally symmetric sets.  We are also grateful to Vladimir Shpilrain for pointing out the connection between our work and the questions of Vladimir Lin.

%%%
%%%
%%%

\section{Totally symmetric sets}\label{sec:tss}

In this section we introduce the main new technical tool in the paper, namely, totally symmetric sets.  After giving some examples, we derive some basic properties of totally symmetric sets, in particular developing the relationship with canonical reduction systems.  The main results in this section are the fundamental lemma for totally symmetric sets (Lemma~\ref{lem:tss}) and a classification of certain totally symmetric subsets of $\B_n$ (Lemma~\ref{lem:bnclass}).  

\p{Totally symmetric subsets of groups} Let $X$ be a subset of a group $G$.  We may conjugate $X$ by an element $g$ of $G$, meaning that we conjugate each element of $X$ by $g$.  We say that $X$ is \emph{totally symmetric} if
\begin{itemize}
\item the elements of $X$ commute pairwise and
\item each permutation of $X$ can be achieved via conjugation by an element of $G$.
\end{itemize}
As a first example, any singleton $\{x\}$ is totally symmetric.  Another example is the set of transpositions
\[
\{(1\ \ 2),(3\ \ 4),\dots,(m\ \ m+1)\}
\]
in the symmetric group $\Sigma_n$, where $m$ is an odd integer less than $n$.

We say that a totally symmetric set $X \subseteq G$ is totally symmetric with respect to a subgroup $H$ of $G$ if $X$ satisfies the above definition, with the additional constraint that the conjugating elements can be chosen to lie in $H$.  We observe that if $X \subseteq G$ is totally symmetric with respect to $H \leqslant G$, and $X \subseteq H$, then $X$ is a totally symmetric subset of $H$.  

The definition of a totally symmetric set is inspired by the work of Aramayona--Souto, who studied a particular example of a totally symmetric set in their work on homomorphisms between mapping class groups; see their paper \cite[Section 5]{AS}.  

\p{New totally symmetric sets from old} Let $X =\{x_1,\dots,x_m\}$ be a totally symmetric subset of $G$.  There are several ways of obtaining new totally symmetric sets from $X$.  Let $k,\ell \in \Z$ and let $z$ be an element of $G$ with the property that each permutation of $X$ can be achieved by an element of $G$ that commutes with $z$ (for example $z$ can lie in $Z(G)$).  Also, for each $i$ let $x_i^*$ denote the product of the $x_j \in X$ with $j \neq i$.  Starting from $X$, we may create the following totally symmetric sets: 
\begin{align*}
X^k &= \{ x_1^k ,\dots, x_m^k \} \\
X^* &= \{x_1^*,\dots,x_m^*\} \\
X^{k,\ell} &= \{x_1^k(x_1^*)^\ell,\dots,x_m^k(x_m^*)^\ell\} \\
X' &= \{x_1x_2^{-1},\dots,x_1x_m^{-1}\} \\
X^z &= \{x_1z,\dots,x_mz \}.
\end{align*}
We can combine these constructions, for instance $(X^k)^z$ and $(X^*)^*$ are totally symmetric.  Also, if all permutations of $X$ are achievable by elements of a subgroup $H$ of $G$, then the same is true for $X^k$, $X^*$, $X'$, and $X^z$.

\p{The fundamental lemma} We have the following fundamental fact about totally symmetric sets.  It is an analog of Schur's lemma from representation theory.  

\begin{lemma}
\label{lem:tss}
Let $X$ be a totally symmetric subset of a group $G$ and let $\rho : G \to H$ be a homomorphism of groups.  Then $\rho(X)$ is either a singleton or a totally symmetric set of cardinality $|X|$.
\end{lemma}

\begin{proof}

It is clear from the definition of a totally symmetric set that $\rho(X)$ is totally symmetric and that its cardinality is at most $|X|$.  Suppose that the restriction of $\rho$ to $X$ is not injective; say $\rho(x_1)=\rho(x_2)$.  For any $x_i \in X$ there is (by the definition of a totally symmetric set) a $g \in G$ so that $(gx_1g^{-1},gx_2g^{-1}) = (x_1,x_i)$.  Thus
\[
\rho(x_1x_i^{-1}) = \rho((gx_1g^{-1})(gx_2^{-1}g^{-1})) = \rho(g)\rho(x_1x_2^{-1})\rho(g)^{-1} = 1.
\]
The lemma follows.
\end{proof}

\p{Totally symmetric sets in braid groups} In the braid group $\B_n$ the most basic example of a totally symmetric set is
\[
X_n = \{\sigma_1,\sigma_3,\sigma_5,\dots,\sigma_{m}\}
\]
where $m$ is the largest odd integer less than $n$.  As above, the sets $X_n^k$, $X_n^*$, $X_n'$, and $X^z$ are totally symmetric.  In the following lemma, let $z \in \B_n$ be a generator for the center $Z(B_n)$; the signed word length of $z$ is $n(n-1)$.  We also define
\begin{align*}
Y_n &= X_n', \text{ and} \\
Z_n &= \left(X_n^{n(n-1)}\right)^{z^{-1}}.
\end{align*}

\begin{lemma}
Let $n \geq 2$.  The set $X_n \subseteq \B_n$ is totally symmetric with respect to $\B_n'$.  In particular, $Y_n$ and $Z_n$ are totally symmetric subsets of $\B_n'$.  
\end{lemma}

\begin{proof}

Suppose some permutation $\tau$ of $X_n$ is achieved by $g \in \B_n$.  Since $\sigma_1$ commutes with each element of $X_n$, the permutation $\tau$ is also achieved by $g\sigma_1^k$ for all $k \in \Z$.  If we take $k$ to be the negative of the signed word length of $g$, then $g\sigma_1^k$ lies in $\B_n'$.  The first statement follows.  The second statement follows similarly, once we observe that each element of $Y_n$ and $Z_n$ lies in $\B_n'$.  
\end{proof}

The goal of the remainder of the section is to classify certain totally symmetric subsets of size $\lfloor n/2 \rfloor$ in $\B_n$.  The end result is Lemma~\ref{lem:bnclass} below.  The proof requires three auxiliary results, Lemmas~\ref{lem:multiclass}, \ref{lem:symcrs}, and~\ref{lem:matrix}.

\p{Totally symmetric multicurves}   The first tool is a topological version of total symmetry.  Let $Y$ be a set and let $S$ be a surface.  We say that a multicurve $M$ is \emph{$Y$-labeled} if each component of $M$ is labeled by a non-empty subset of $Y$.  The symmetric group $\Sigma_Y$ acts on the set of $Y$-labeled multicurves by acting on the labels.  The mapping class group $\Mod(S)$---the group of homotopy classes of orientation-preserving homeomorphisms of $S$ fixing the boundary of $S$---also acts on the set of $Y$-labeled multicurves via its action on the set of multicurves.  

Let $M$ be a $Y$-labeled multicurve in $S$.  We say that $M$ is \emph{totally symmetric} if for every $\sigma \in \Sigma_Y$ there is an $f \in \Mod(S)$ so that $\sigma \cdot M = f \cdot M$.  As in the case of totally symmetric sets, we say that $M$ is totally symmetric with respect a subgroup $H$ of $\Mod(S)$ if the elements $f$ from the definition can all be chosen to lie in $H$.

We say that a $Y$-labeled multicurve has the \emph{trivial labeling} if each component of the multicurve has the label $Y$ (recall that empty labels are not allowed).  Every such multicurve is totally symmetric (with respect to any subgroup $H$ of $\Mod(S)$).  We also say that a component of a $Y$-labeled multicurve has the trivial label if its label is $Y$.

We can describe a $Y$-labeled multicurve in a surface $S$ as a set of pairs $\{(d_i,A_i)\}$ where each $d_i$ is a curve in $S$, where each $A_i$ is a subset of $Y$, and where $\{d_i\}$ is a multicurve in $S$.

Let $Y$ be a set.  If $M=\{(d_1,A_1),\dots,(d_m,A_m)\}$ is a totally symmetric $Y$-labeled multicurve in a surface $S$, then we may create new totally symmetric multicurves from $Y$ as follows.  For a subset $A$ of $Y$, we denote by $A^c$ the complement $Y \setminus A$.  The new totally symmetric multicurve is
\begin{align*}
M^* &= \{(d_1,A_1^c),\dots,(d_m,A_m^c)\}
\end{align*}

Let $N = \lfloor n/2 \rfloor$, and let $[N]$ denote the set $\{1,\dots,N\}$.  For $1 \leq i \leq N$ let $c_i$ be the curve in $\D_n$ with the property that the half-twist about $c_i$ is $\sigma_{2i-1}$. The $[N]$-labeled multicurve
\[
M_n = \{(c_1,\{1\}),(c_2,\{2\}),\dots,(c_N,\{N\})\}
\]
is totally symmetric.  

For the statement of the next lemma, we require several further definitions.  First, for $H$ a subgroup of $\Mod(S)$, we say that two labeled multicurves in $S$ are $H$-equivalent if they lie in the same orbit under $H$.  

Next, let $c_0$ denote the standard  curve in $\D_n$ that surrounds the first $n-1$ marked points (so that $c_0$ is disjoint from $c_1,\dots,c_N$). The multicurve in $\D_n$ (with $n$ odd) whose components are $c_0,\ldots, c_N$ is depicted in Figure~\ref{fig:mc}.  For $n$ odd, let $\widehat M_n$ and $\widehat M_n^*$ be the labeled multicurves $\widehat M_n = M_n\cup\{c_0,[N]\}$ and $\widehat M_n^* = M_n^*\cup\{c_0,[N]\}$; these are depicted in Figure~\ref{fig:lmc}.

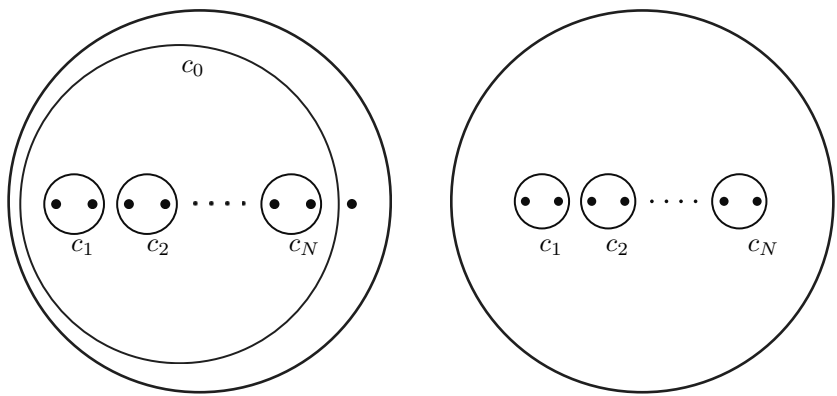
\begin{figure}
\labellist
\small\hair 2pt
\pinlabel $c_1$ at 57 110
\pinlabel $c_2$ at 114 110
\pinlabel $c_{N}$ at 224 110
\pinlabel $c_0$ at 140 245
\pinlabel $c_1$ at  410 110
\pinlabel $c_2$ at 460 110
\pinlabel $c_{N}$ at 570 110
\endlabellist
\includegraphics[scale=.5]{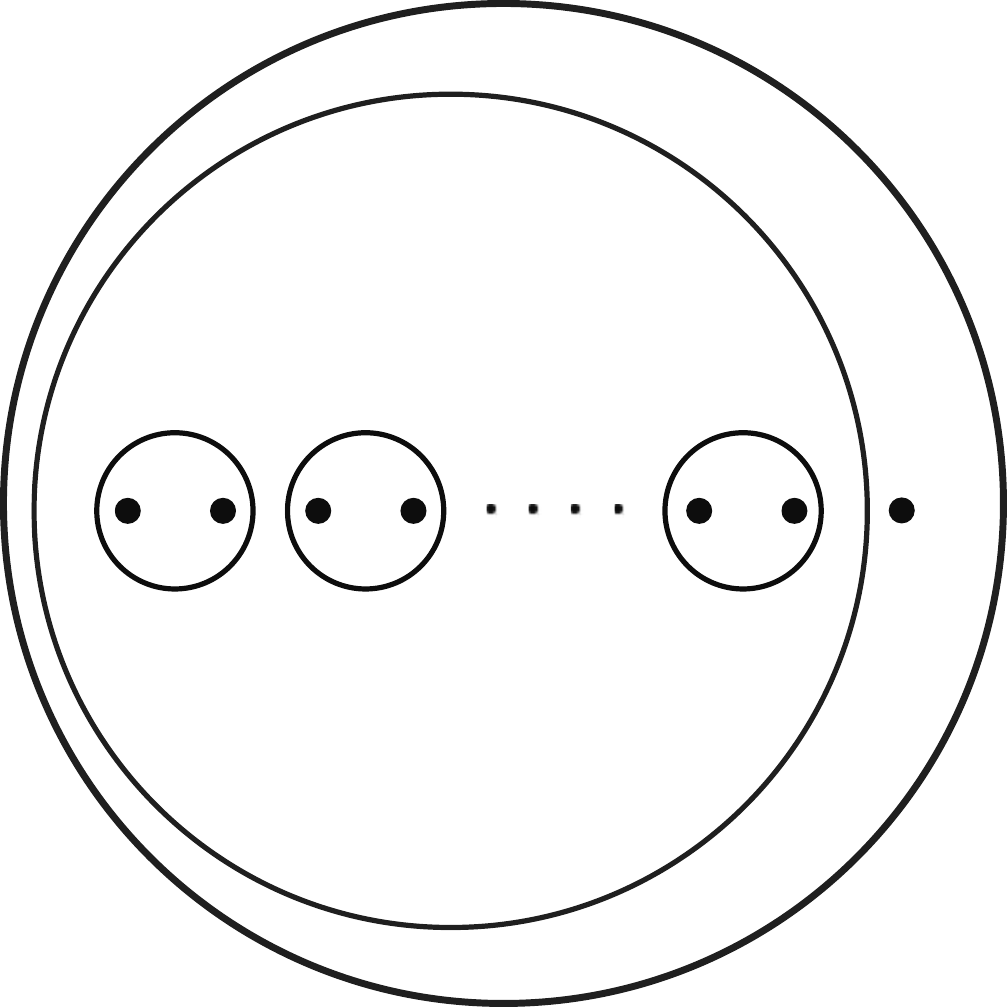}\quad\quad
\includegraphics[scale=.5]{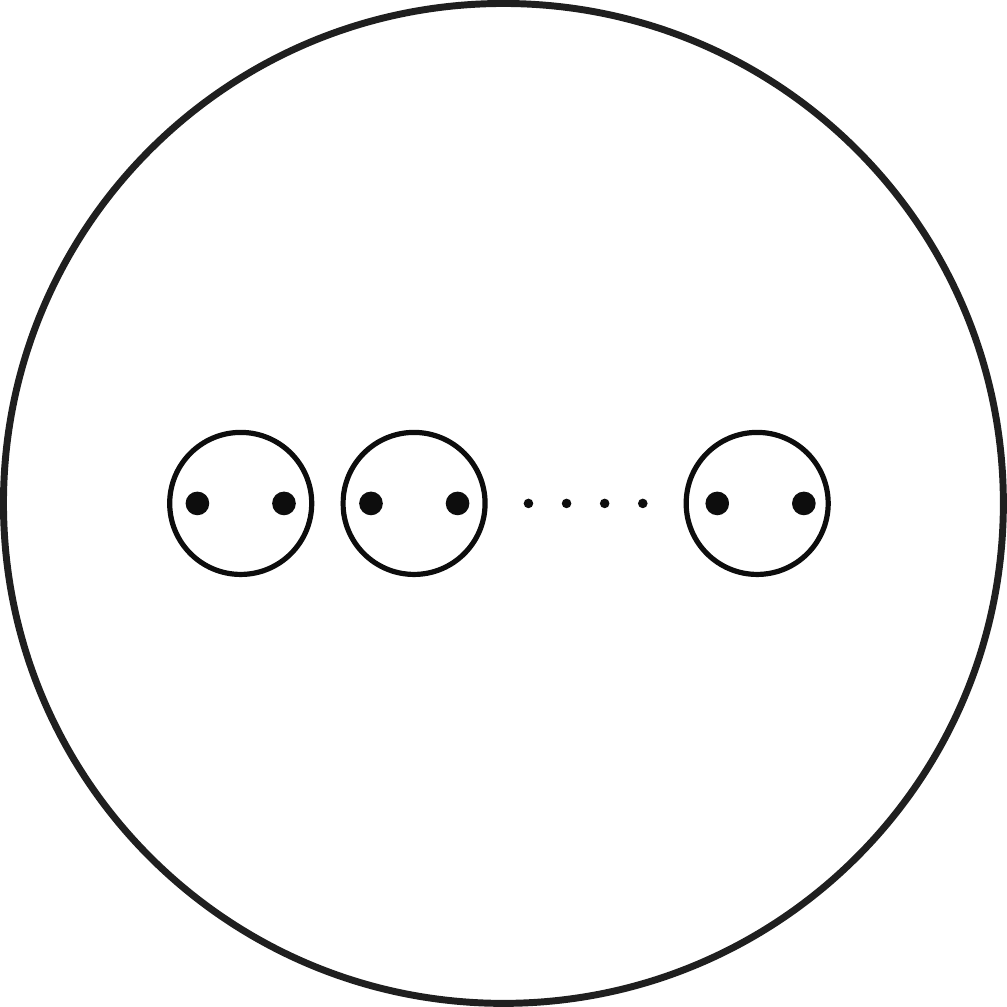}
\caption{\emph{Left:} the curves $c_0,\ldots, c_N$ in $\D_n$ with $n$ odd; \emph{Right:} the curves $c_1,\ldots, c_N$ in $\D_n$ with $n$ even}
\label{fig:mc}
\end{figure}

\begin{lemma}
\label{lem:multiclass}
Let $n \geq 1$, let $N = \lfloor n/2 \rfloor$.  
\begin{enumerate}
\item If $n$ is even, then every totally symmetric $[N]$-labeled multicurve in $\D_n$ with nontrivial labeling is $\B_n$-equivalent to $M_n$ or  $M_n^*$.
\item If $n$ is odd, then every totally symmetric $[N]$-labeled multicurve in $\D_n$ with nontrivial labeling is $\B_n$-equivalent to $M_n$, $M_n^*$, $\widehat M_n$, or  $\widehat M_n^*$.
\end{enumerate}
\end{lemma}

\begin{proof}

Say that $M$ is a totally symmetric $[N]$-labeled multicurve in $\D_n$ with nontrivial label.  Let $c$ be a curve in $M$ with nontrivial label $\emptyset \neq A \subsetneq [N]$.  The symmetric group $\Sigma_N$ acts on the power set of $[N]$ and the orbit of $A$ under this action has $\ell \geq N$ elements.  Since $M$ is totally symmetric, there must be for each element $A'$ of this orbit a curve $d$ in $M$ so that: (1) $d$ lies in the same $\B_n$-orbit as $c$ and (2) the label for $d$ is $A'$.  In particular, $M$ contains distinct curves $d_1,\dots,d_\ell$ that all lie in the same $\B_n$-orbit.  It follows that $\ell=N$, that each $d_i$ surrounds exactly two marked points, and that the labels are either of the form $\{i\}$ or $\{i\}^c$.  We can further conclude that there are no other curves in $M$ with nontrivial label besides $d_1,\dots,d_N$.  Up to $\B_m$-equivalence, we may therefore assume that the labeled multicurve $\{d_1,\dots,d_N\}$ is exactly $M_n$ or $M_n^*$.  

Let $T = \{b_1,\dots,b_k\}$ be the set of curves of $M$ with trivial label.  The curves of $T$ induce a partition of the set $\{d_1,\dots,d_N\}$: the curves $d_i$ and $d_j$ are in the same subset of the partition if and only if they are not separated by an element of $T$.  Since the $b_i$ are essential and distinct from the $d_i = c_i$, it must be that either (1) $k=1$ and (up to $\B_m$-equivalence) $b_1=c_0$ or (2) the partition is nontrivial, meaning that it contains more than one subset.  The second possibility violates the assumption that $M$ is totally symmetric.  The lemma follows.
\end{proof}

\p{From totally symmetric sets to totally symmetric multicurves} Associated to each element $f$ of $\Mod(S)$ is its canonical reduction system $\CRS(f)$, which is a multicurve.  We will make use of several basic facts about canonical reduction systems.  First, we have $\CRS(f) = \emptyset$ if and only if $f$ is periodic or pseudo-Anosov.  Next, if $f$ and $g$ commute then $\CRS(f) \cap \CRS(g) = \emptyset$.  Also, for any $f$ and $g$ we have $\CRS(gfg^{-1}) = g\CRS(f)$.  See the paper by Birman--Lubotzky--McCarthy for background on canonical reduction systems \cite{blm}.

Given a totally symmetric subset $X = \{x_1,\dots,x_m\}$ of $\Mod(S)$ we obtain an $[m]$-labeled multicurve as follows: the underlying multicurve $M$ is obtained from the disjoint union of the $\CRS(x_i)$ by identifying homotopic curves, and the label of a curve $c \in M$ is the set of $i$ with $c$ a component of $\CRS(x_i)$.  We denote this $[m]$-labeled multicurve by $\CRS(X)$.  We have the following lemma, which follows immediately from the definitions and the stated facts about canonical reduction systems.

\begin{lemma}
\label{lem:symcrs}
If $X$ is a totally symmetric subset of $\Mod(S)$ then $\CRS(X)$ is totally symmetric.
\end{lemma}

The totally symmetric multicurves associated to $X_n$, $Y_n$, and $Z_n$ are
\begin{align*}
\CRS(X_n) &= \{(c_1,\{1\}),(c_2,\{2\}),\dots,(c_N,\{N\})\} = M_n \\
\CRS(Y_n) &= \{(c_1,[N])\} \cup \{(c_2,\{1\}),\dots,(c_N,\{N-1\}\}, \text{ and}\\
\CRS(Z_n) &= \{(c_1,\{1\}),\dots,(c_N,\{N\})\} = M_n.
\end{align*}

\begin{figure}
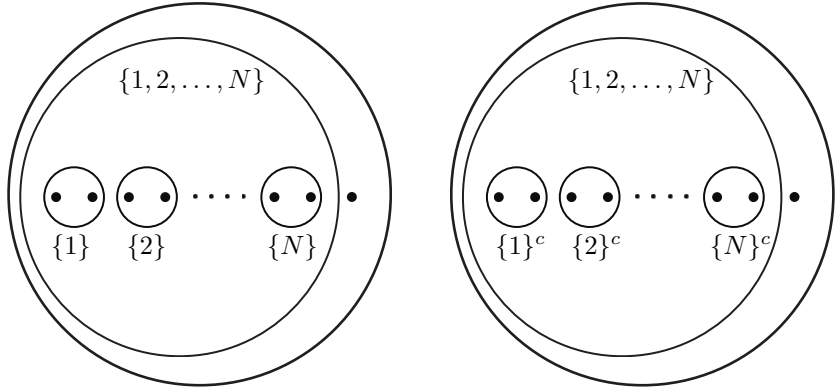

\labellist
\small\hair 1pt
\pinlabel $\{1\}$ at 48 105
\pinlabel $\{2\}$ at 105 105
\pinlabel $\{N\}$ at 215 105
\pinlabel $\{1,2,\dots,N\}$ at 139 230
\pinlabel $\{1\}^c$ at 387 105
\pinlabel $\{2\}^c$ at 444 105
\pinlabel $\{N\}^c$ at 553 105
\pinlabel $\{1,2,\dots,N\}$ at 478 230
\endlabellist
\includegraphics[scale=.5]{Mnc0}\quad\quad
\includegraphics[scale=.5]{Mnc0}
\caption{\emph{Left:} the labeled multicurve $\widehat M_n$; \emph{Right:} the labeled multicurve $\widehat M_n^*$}
\label{fig:lmc}
\end{figure}

\p{Classification of derived totally symmetric subsets}  Let $X = \{x_1,\dots,x_m\}$ be a totally symmetric subset of a group $G$.  In this case,  we say that a totally symmetric set $Y$ in $G$ is \emph{derived} from $X$ if $Y$ lies in the free abelian subgroup $\langle X \rangle$ of $G$.  We have already seen examples of derived totally symmetric sets, such as $X^k$, $X^*$, $X^{k,\ell}$, and $X'$.  

Let $X$ be a totally symmetric subset of a group $G$, and let $Y$ be a derived totally symmetric subset.  We consider the action of $G$ on itself by conjugation and write $\Stab_G(X)$ and $\Stab_G(Y)$ for the stabilizers of the sets $X$ and $Y$.    We say that the derived totally symmetric set $Y$ is \emph{robust} if
\[
\Stab_G(Y) \subseteq \Stab_G(X).
\]
As an example in $G = \B_n$, the totally symmetric set $Y = X_n^{k,\ell}$ is a robust totally symmetric set in $X_n$ as long as at least one of $k$ and $\ell$ is nonzero.  Indeed, since each $c_i$ lies in the canonical reduction system for some element of $X_n^{k,\ell}$, any element of $\Stab_G(X_n^{k,\ell})$ preserves the set of curves $\{c_1,\dots,c_N\}$ and hence lies in $\Stab_G(X_n)$.  

\begin{lemma}
\label{lem:matrix}
Let $X = \{x_1,\dots,x_m\}$ be a totally symmetric subset of a group $G$, and let $Y$ be a robust derived totally symmetric set with $m$ elements.  Then $Y$ is equal to some $X^{k,\ell}$.  
\end{lemma}

\begin{proof}

We may assume that $m \geq 2$, for otherwise the lemma is trivial.  Say that the elements of $Y$ are $y_1,\dots,y_m$ and that the elements of $X$ have order $d$.  Then the elements of $Y$ can be written as
\[
y_i = x_1^{a_{i,1}} \cdots x_m^{a_{i,m}},
\] 
where each $a_{i,j}$ lies in $\Z/d\Z$ (when $d = \infty$ we interpret $\Z/d\Z$ as $\Z$).  Let $A$ be the $m\times m$ matrix $\left(a_{i,j}\right)$.  As such, the $i$th row of $A$ records the exponents on the $x_j$ in the expression for $y_i$. 

The statement of the lemma is equivalent to the statement that there exist $k,\ell \in \Z/d\Z$ such that, up to reordering the rows of $A$, we have

\medskip

\[
 A = 
 \begin{pmatrix}
 k & \ell & \cdots & \ell \\
 \ell & k  & \cdots & \ell \\
 \vdots & \vdots & \ddots &\vdots\\
 \ell & \ell & \cdots & k
 \end{pmatrix}
\]

\medskip

We claim that any permutation of the rows of $A$ is achieved by a permutation of the columns of $A$. Let $\sigma\in \Sigma_m$ and let $g\in G$ be such that $gy_ig^{-1} = y_{\sigma(i)}$. This determines a permutation of the rows of $A$.  By the total symmetry of $Y$, every permutation of the rows arises in this way. On the other hand, since $Y$ is robust, the conjugating element $g$ also permutes $X$, and hence determines a permutation of the columns of $A$. As both permutations have the same effect on the set $Y$, the claim follows.

We next claim that if $v$ is a column of $A$, and $w$ is an element of $(\Z/d\Z)^m$ obtained by permuting the entries of $v$, then $w$ is also a column of $A$.  The claim follows from the previous claim and the fact that any permutation of the entries of $v$ can be achieved by a permutation of the rows of $A$.

It must be that some column of $A$ has at least two distinct entries; if not, then the rows of $A$ are equal, violating that assumption that the $y_i$ are distinct.  Let $v$ be such a column of $A$.   It must be that, up to reordering the rows of $A$, we have $v = (k,\ell,\dots,\ell)$ for some $k$ and $\ell$.  Indeed, otherwise, there would be more than $m$ distinct permutations of the entries of $v$, violating the previous claim.  It further follows from the previous claim that the $m$ columns of $A$ are the $m$ distinct permutations of the entries of $v$.  After reordering the rows, $A$ has the desired form.  
\end{proof}

\p{Classification of totally symmetric subsets of the braid group} In the following lemma, we say that two totally symmetric subsets of a group $G$ are \emph{$G$-equivalent} if there is an automorphism of $G$ taking one to the other.  

\begin{lemma}
\label{lem:bnclass}
Let $n \geq 5$, let $N = \lfloor n/2 \rfloor$, and let $X=\{x_1,\dots,x_N\}$ be a totally symmetric subset of $\B_n$.  Assume that $\CRS(X)$ has a nontrivial labeling.
\begin{enumerate}[itemsep=.5em]
\item If $n$ is even, $X$ is $\B_n$-equivalent to $(X_n^\ell)^{z^s}$ or $((X_n^*)^\ell)^{z^s}$ for some $k,s \in \Z$ with $\ell \neq 0$.  
\item If $n$ is odd, $X$ is $\B_n$-equivalent to $(X_n^\ell)^{T_{c_0}^rz^s}$ or $((X_n^*)^\ell)^{T_{c_0}^rz^s}$ for some $\ell,r,s \in \Z$ with $\ell \neq 0$. 
\end{enumerate}
\end{lemma}

\begin{proof}[Proof of Lemma~\ref{lem:bnclass}]

We treat the case where $n$ is odd.  The other case is similar (and simpler).  By Lemma~\ref{lem:symcrs}, the multicurve $\CRS(X)$ is totally symmetric.  It then follows from Lemma~\ref{lem:multiclass} that $\CRS(X)$ is $\B_n$-equivalent to $M_n$, $M_n^*$, $\widehat M_n$, or  $\widehat M_n^*$.  We may assume then that $\CRS(X)$ is in fact equal to $M_n$, $M_n^*$, $\widehat M_n$, or  $\widehat M_n^*$.  We discuss the cases $\widehat M_n$ and  $\widehat M_n^*$ in turn, the other cases again being similar and simpler.

Suppose first that $\CRS(X)$ is equal to $\widehat M_n$.  This means that $\CRS(x_1)$ is equal to $\{c_0,c_1\}$.  The multicurve $\CRS(x_1)$ divides $\D_n$ into three regions, each corresponding to a periodic or pseudo-Anosov Nielsen--Thurston component of $x_1$.  To prove the lemma in this case it suffices to show that the outer two Nielsen--Thurston components of $x_1$ are trivial.  Indeed, it follows from this that $x_1$ is of the desired form $H_{c_1}^\ell T_{c_0}^r z^s$ with $\ell \neq 0$, and then by the total symmetry that each $x_i$ is of the desired form $H_{c_i}^\ell T_{c_0}^r z^s$.

The outermost Nielsen--Thurston component of $x_1$ (exterior to $c_0$) is necessarily trivial since this outermost region is a pair of pants (after collapsing the boundary components to points), and since $x_1$ fixes the three marked points.

It remains to show that the Nielsen--Thurston component of $x_1$ corresponding to the region lying between $c_0$ and $c_1$ is trivial.  Since $x_1$ commutes with $x_2$, it follows that $x_1$ fixes $\CRS(x_2) = \{c_2\}$.  From this, it immediately follows that the Nielsen--Thurston component in question is not pseudo-Anosov.  Therefore it is periodic.  If we collapse $c_0$ and $c_1$ to marked points, the region lying between $c_0$ and $c_1$ becomes a sphere with $n-1$ marked points.  A periodic mapping class of this sphere is a rotation fixing the marked points coming from $c_0$ and $c_1$.  Since $x_1$ fixes $c_2$ (which surrounds two marked points) it follows that the rotation is trivial, completing the proof of the lemma in this case.

We now address the case where $\CRS(X)$ is equal to $\widehat M_n^*$.  In this case, it follows as in the previous case that each $x_i$ is of the form
\[
x_i = P_i T_{c_0}^rz^s
\]
where each $P_i$ is a product of nonzero powers of half-twists about the elements of $\{c_1,\dots,c_N\} \setminus \{c_i\}$ (it follows from the total symmetry that $r$ and $s$ are independent of $i$).  Let $Y$ be the totally symmetric set $Y=X^{T_{c_0}^{-r}z^{-s}}$.  This $Y$ has $N$ elements and is a robust totally symmetric set derived from $X_n$  (the argument is the same as the above argument that $X_n^{k,\ell}$ is robust in $X_n$).  By Lemma~\ref{lem:matrix}, $Y$ is equal to $X^{k,\ell}$ for some $k, \ell \in \Z$.  It then follows from the fact that $\CRS(X) = \widehat M_n^*$ that $k=0$ and $\ell \neq 0$.  This implies that $Y = (X_n^*)^\ell$.  It follows that
\[
X = Y^{T_{c_0}^rz^s} = ((X_n^*)^\ell)^{T_{c_0}^rz^s},
\]
as desired.
\end{proof}

%%%
%%%
%%%

\section{Proof of Theorem~\ref{thm:main}}
\label{sec:proof}

In this section, we prove Theorem~\ref{thm:main}, which states that, for $n \geq 7$, every homomorphism $\rho : \B_n' \to B_n$ is equivalent to the inclusion map.  We require the following tool, a direct product decomposition for certain subgroups of $\B_n$.  

\p{The cabling decomposition} Let $M$ be a multicurve in $\D_n$.  Our next goal is to give an iterated semi-direct product decomposition for $\Stab_{\B_n}(M)$, the stabilizer of $M$ in $\B_n$.   The desired decomposition is given in Lemma~\ref{lemma:cabling} below.

We begin with some setup.  Let $\Gamma$ be the graph defined as follows: there is a distinguished vertex corresponding to the boundary of $\D_n$ and there are other vertices corresponding to the components of $M$; the edges correspond to vertices that are adjacent in the sense that there is a path in $\D_n$ from one curve to the other that does not intersect any other component of $M$.  

Since $\D_n$ is a planar surface, $\Gamma$ is a tree, and we may think of it as a rooted tree with the distinguished vertex as the root.  For $i \geq 1$ let $M_i$ be the multicurve consisting of components of $M$ that have distance $i$ from the root in $\Gamma$, and let $M_0$ denote $\partial \D_n$.  The multicurve $M_1$ consists of the outermost components of $M$.

The group $\Stab_{\B_n}(M)$ acts on the rooted tree $\Gamma$ by simplicial automorphisms.  In particular, it fixes $M_0$ and acts on each $M_i$.  A key property of each $M_i$ is that it is an un-nested multicurve, that is, no component is contained in the disk bounded by another component.  

Let $\Delta_0$ be the disk with marked points obtained from $\D_n$ by crushing to a marked point each disk bounded by a component of $M_0$.  We label each of the marked points by the number of marked points in the interior of the corresponding component of $M_0$.  Some of the marked points of $\Delta_0$ correspond directly to the marked points of $\D_n$ that lie in the exterior of $M_0$; we give those marked points the label 1.  Let $\B_{\Delta_0}$ denote the subgroup of the mapping class group of $\Delta_0$ that preserves the labels. If $\Delta_0$ has $\ell$ marked points, then $\B_{\Delta_0}$ is the subgroup of $\B_\ell$ that preserves the labels of the strands.

Say that $M_0$ has $q$ components, which surround $n_1,\dots,n_q$ marked points, respectively.  There is a map 
\[
\Pi : \Stab_{\B_n}(M_0) \to \B_{\Delta_0}  
\]
obtained by crushing the disks bounded by the components of $M_0$ to marked points.  Since $M_0$ is un-nested, the kernel of $\Pi$ is isomorphic to the direct product  
\[
\B_{n_1} \times \cdots \times \B_{n_q}.
\]
The map $\Pi$ is split (the argument is the same as the one used in the paper by Lei Chen and the authors of this paper \cite[Lemma 6.2]{chenkordekmargalit}).  We thus have a semi-direct product
\[
\Stab_{\B_n}(M_0) \cong \B_{\Delta_0} \ltimes \left( \B_{n_1} \times \cdots \times \B_{n_q} \right).
\]
For each $j$, let $M_{0,j}$ be the multicurve in $\D_{n_j}$ corresponding the children in $\Gamma$ of the $j$th component of $M_0$. It follows from this semi-direct product decomposition of $\Stab_{\B_n}(M_0)$ that we have the following semi-direct product decomposition for $\Stab_{\B_n}(M)$.

\begin{lemma}
\label{lemma:cabling}
Let $M$ be a multicurve in $\D_n$.  Let $M_0$ be the multicurve in $\D_n$ consisting of the outermost components of $M$, and let $M_0$ and $M_{0,1},\dots,M_{0,q}$ be the multicurves in $\D_{n_1},\dots,\D_{n_q}$ defined above.  The map $\Pi$ induces a semi-direct product decomposition
\[
\Stab_{\B_n}(M) \cong \B_{\Delta_0} \ltimes \left( \Stab_{\B_{n_1}}(M_{0,1}) \times \cdots \times \Stab_{\B_{n_q}}(M_{0,q}) \right).
\]
\end{lemma}

We remark that Lemma~\ref{lemma:cabling} can be iteratively applied in order to decompose each of the $\Stab_{\B_{n_j}}(M_{0,j})$, giving an iterated semi-direct product decomposition of $\Stab_{\B_n}(M)$.  In the final decomposition, each factor is a subgroup of some braid group, specifically, a subgroup preserving some partition of the strands.

\begin{proof}[Proof of Theorem~\ref{thm:main}]

As in the statement, assume $n \geq 7$ and let $\rho : \B_n' \to \B_n$ be a nontrivial homomorphism.  Let $Z_n$ be the totally symmetric subset of $\B_n'$ defined in Section~\ref{sec:tss}, and let $M$ denote the labeled multicurve $\CRS(\rho(Z_n))$.  For $1 \leq i \leq \lfloor n/2 \rfloor$ let $x_i$ be the element
\[
x_i = \rho(\sigma_i^{p}z^{-1}),
\]
so that $\rho(Z_n)$ is the set of all $x_i$ (they may not be all distinct).  

By Lemma~\ref{lem:tss}, the cardinality of $\rho(Z_n)$ is either 1 or $\lfloor n/2 \rfloor$.  We thus have four cases:
\begin{enumerate}
\item $\rho(Z_n)$ is a singleton,
\item $M$ is empty,
\item $M$ is non-empty and is trivially labeled, and
\item $M$ is non-empty and is not trivially labeled.
\end{enumerate}
In the first two cases we will show that $\rho$ is trivial, in the third case we will derive a contradiction, and in the last case we will show that $\rho$ is equivalent to the inclusion map.

\bigskip

\noindent \emph{Case 1.}  If $\rho(Z_n)$ is a singleton then $x_1=x_3$ and so
\[
\rho(\sigma_1\sigma_3^{-1})^{p} = \rho((\sigma_1\sigma_3^{-1})^{p}) = \rho((\sigma_1^{p}z^{-1})(\sigma_3^{p}z^{-1})^{-1}) = x_1x_3^{-1} = 1.
\]
Since $\B_n$ is torsion-free, we therefore have that $\rho(\sigma_1\sigma_3^{-1}) = 1$, and since the normal closure of $\sigma_1\sigma_3^{-1}$ in $\B_n'$ is equal to $\B_n'$ for $n \geq 5$ (see~\cite[Lemma 8.4]{chenkordekmargalit}) it follows that $\rho$ is trivial.  

\bigskip

\noindent \emph{Case 2.}  In this case we will prove that $\rho$ is trivial. To say that $M$ is empty is to say that the $x_i$ are periodic or they are pseudo-Anosov.  

Assume first that the $x_i$ are periodic.  Since the image of $\rho$ is contained in $\B_n'$ and since the only periodic element of $\B_n'$ is the identity, we must have that $x_i = 1$ for all $i$. This implies that $\rho(Z_n)$ is a singleton, in which case we can apply Case 1 in order to conclude that $\rho$ is trivial. 

Now assume that the $x_i$ are all pseudo-Anosov. Let $\bar \B_n$ denote the quotient $\B_n/Z(\B_n)$ and let $\bar x_i$ denote the image of $x_i$ in this quotient. Since the $x_i$ are pseudo-Anosov, so too are the $\bar x_i$. By a theorem of McCarthy~\cite{mccarthy}, there is a short exact sequence
\[
1 \to F \to C_{\bar \B_n}(\bar x_1) \to \Z \to 1,
\]
where $C_{\bar \B_n}(\bar x_1)$ is the centralizer of $x_1$ in $\bar \B_n$ and $F$ is a finite subgroup of $\bar \B_n$.  Since $\bar x_3$ commutes with $\bar x_1$ and is conjugate, it follows that $(\bar x_1 \bar x_3^{-1})^p$ lies in $F$ for some $p$.  In particular $\bar x_1 \bar x_3^{-1}$, hence $x_1x_3^{-1}$, is periodic.  We again use the fact that $\B_n'$ contains no nontrivial periodic elements to conclude that $x_1 x_3^{-1} = 1$.  By the fundamental lemma of totally symmetric sets (Lemma~\ref{lem:tss}), the set $\rho(Z_n)$ is a singleton and we may again apply Case 1 in order to conclude that $\rho$ is trivial, contradicting the assumption that the $x_i$ are pseudo-Anosov.

\bigskip

\noindent \emph{Case 3.} The basic strategy is to show that $\rho(\B_n')$ lies in $\Stab_{\B_n}(M)$, to decompose the latter using Lemma~\ref{lemma:cabling}, and then to inductively apply the argument of Case 2 to the resulting factors in order to derive a contradiction.  

For $n \geq 5$, the group $\B_n'$ is generated by the elements of the form $\sigma_1\sigma_i^{-1}$ where $2\leq i \leq n-1$; see \cite[p. 7]{linbp} or \cite[Prop. 3.1]{kordek}. For each such generator $\sigma_1\sigma_i^{-1}$, there exists a $\sigma_j^pz^{-1} \in Z_n$ that commutes with it.  This implies that each $\rho(\sigma_1\sigma_i^{-1})$ preserves $\Gamma(\rho(\sigma_j^pz^{-1}))$.  Since $M$ is trivially labeled, the latter is equal to $M$. Thus $\rho(\B_n')$ lies in $\Stab_{\B_n}(M)$.

Let $\Pi$ be the map from Lemma~\ref{lemma:cabling}.  The composition $\Pi \circ \rho : \B_n' \to \B_{\Delta_0}$ satisfies the hypothesis of Case 2, in that $\Gamma(\Pi \circ \rho (Z_n))$ is empty.  By the argument of Case 2, $\Pi \circ \rho$ is trivial (we cannot apply Case 2 verbatim since $\Delta_0$ has fewer than $n$ marked points).  Thus, the image of $\rho$ lies in the second factor of the decomposition of $\Stab_{\B_n}(M)$ given by Lemma~\ref{lemma:cabling}.  We may now iterate the argument on each factor.  After finitely many steps, we conclude that $\rho$ is trivial, a contradiction.  

\bigskip

\noindent \emph{Case 4.} In this case we will prove that $\rho$ is equivalent to the identity. The proof has five steps.  In the fourth step, we say that a sequence of curves $d_1,\dots,d_k$ in $\D_n$ forms a chain if each $d_i$ surrounds two marked points, if $i(d_i,d_{i+1})=2$ for $1 \leq i \leq k-1$, and if the $d_i$ are disjoint otherwise.  Also, let $a_1,\dots,a_{n-1}$ be the curves in $\D_n$ with the property that $\sigma_i = H_{a_i}$ for $1 \leq i \leq n-1$.  Note that $c_i = a_{2i-1}$ for $1 \leq i \leq \lfloor n/2 \rfloor$ and that $a_1,\dots,a_{n-1}$ form a chain.

%\medskip

\newpage

\begin{itemize}[leftmargin=12.5ex]
\item[\emph{Step 1.}] Up to equivalence, we have $\rho(\sigma_1\sigma_j^{-1}) = \sigma_1\sigma_j^{-1}$ for all odd $j$. 
\item[\emph{Step 2.}] For each even $i\geq 6$ there exists a curve $b_i$ such that $\rho(\sigma_1\sigma_i^{-1}) = \sigma_1H_{b_i}^{-1}$.
\item[\emph{Step 3.}] For $i\in \{2,4\}$ there exists a curve $b_i$ such that $\rho(\sigma_1\sigma_i^{-1}) = \sigma_1H_{b_i}^{-1}$.
\item[\emph{Step 4.}] The curves $a_1,b_2,a_3,b_4,a_5,b_6,\dots$ form a chain.
\item[\emph{Step 5.}] The homomorphism $\rho$ is equivalent to the inclusion map.
\end{itemize}
We complete the five steps in turn.  Step 3 is the only part of the proof that uses the assumption $n \geq 7$; Step 4 uses the assumption $n \geq 6$, and the rest of the proof only uses the assumption $n \geq 5$.  

\bigskip

\noindent \emph{Step 1. Up to equivalence, we have $\rho(\sigma_1\sigma_j^{-1}) = \sigma_1\sigma_j^{-1}$ for all odd $j$.}  

\bigskip

When $n$ is even, Lemma~\ref{lem:bnclass} implies there is a nonzero $\ell$ so that $\rho(Z_n)$ is $\B_n$-equivalent to one of the following totally symmetric sets:
\[
(X_n^{0,\ell})^{T_{c_0}^rz^s}\quad \text{or}\quad (X_n^{\ell, 0})^{T_{c_0}^rz^s}.
\]
When $n$ is odd, Lemma~\ref{lem:bnclass} implies that there is a nonzero $\ell$ so that $\rho(Z_n)$ is $\B_n$-equivalent to one of the following totally symmetric sets:
\[
(X_n^{0,\ell})^{z^s}\quad (X_n^{\ell, 0})^{z^s}\quad (X_n^{0,\ell})^{T_{c_0}^rz^s}\quad \text{or} \quad (X_n^{\ell, 0})^{T_{c_0}^rz^s}.
\]
Therefore, up to replacing $\rho$ by an equivalent homomorphism, we may assume that $\rho(Z_n)$ is equal to one of these sets.

We claim that there exists $q\in \Z$, depending only on $\rho$, such that 
\[
\rho(\sigma_{1}\sigma_{j}^{-1}) = (\sigma_{1}\sigma_{j}^{-1})^{q}.
\]
for all odd $j$ with $1 < j \leq n-1$.  Let $p = n(n-1)$.  Regardless of which of the above sets $\rho(Z_n)$ is equal to, we have
\[
\rho(\sigma_{1}\sigma_{j}^{-1})^p =\rho(\sigma_{1}^pz^{-1})\rho(\sigma_{j}^pz^{-1})^{-1} = (\sigma_{1}\sigma_{j}^{-1})^{\pm\ell}
\]
(in all cases, the $T_{c_0}^r$ and $z^s$ terms cancel each other; the sign of the exponent in the last term depends on whether we have $X_n^{\ell,0}$ or $X_n^{0,\ell}$).  In our earlier paper with Chen~\cite[Lemma 8.7]{chenkordekmargalit}, we proved for any $p$ that $(\sigma_{1}\sigma_{j}^{-1})^{\ell}$ has a $p$th root if and only if $p$ divides $\ell$ and in this case there is a unique root, namely, $(\sigma_{1}\sigma_{j}^{-1})^{\ell/p}$.  Setting $q = \pm \ell/p$ then gives the claim.

We next claim that $q=\pm 1$. Let $g$ be an element of $\B_n'$ such that
\[
g\sigma_1 g^{-1} = \sigma_2
\]
and such that $g$ commutes with each $\sigma_5$ (for example $g = \sigma_2^{-2}\sigma_1\sigma_2$). Then $g$ conjugates $\sigma_1\sigma_5^{-1}$ to $\sigma_2\sigma_5^{-1}$.
Define $b$ to be the curve with
\[
H_{b} = \rho(g)\sigma_1\rho(g)^{-1}.
\]
We then have that
\begin{align*}
\rho(\sigma_2\sigma_5^{-1}) = \rho\left(g(\sigma_1\sigma_5^{-1})g^{-1}\right) =  \rho(g)\rho(\sigma_1\sigma_5^{-1})\rho(g)^{-1} = \rho(g)\left(\sigma_1\sigma_5^{-1}\right)^\ell\rho(g)^{-1} = 
H_{b}^{q}\sigma_5^{-q}.
\end{align*}
The element $\sigma_1\sigma_5^{-1}$ satisfies a braid relation with $\sigma_2\sigma_5^{-1}$, and so $\rho(\sigma_1\sigma_5^{-1})$ satisfies a braid relation with $\rho(\sigma_2\sigma_5^{-1})$. It follows that $(\sigma_1\sigma_5^{-1})^{q}$ satisfies a braid relation with $H_b^q\sigma_5^{-q}$. Since $\sigma_5$ commutes with both $\sigma_1$ and $H_{b}$, we further have that $\sigma_1^q$ satisfies a braid relation with $H_{b}^q$. A result of Bell and second author~\cite[Lemma 4.9]{bellmargalit} states that if two half-twists $H^r_a$ and $H^{s}_b$ satisfy a braid relation, then $r=s = \pm 1$. The claim follows. 

If $q = 1$, then $\rho(\sigma_1\sigma_j^{-1}) = \sigma_1\sigma_j^{-1}$ for $j$ odd, as desired. If $q = -1$, then we may further postcompose $\rho$ with the inversion automorphism of $\B_n$ to obtain again that $\rho(\sigma_1\sigma_j^{-1}) = \sigma_1\sigma_j^{-1}$. This completes the first step.

\bigskip

\noindent \emph{Step 2. For each even $i\geq 6$ there exists a curve $b_i$ such that $\rho(\sigma_1\sigma_i^{-1}) = \sigma_1H_{b_i}^{-1}$.} 

\bigskip

Fix some $i\geq 5$. There exists $g_i\in \B_n'$ such that 
\[
g_i\sigma_5g_i^{-1} = \sigma_i
\] 
and such that $g_i$ commutes with each of $\sigma_1$ and $\sigma_3$; for instance we may take
\[
g_i = \sigma_i^{9-2i} (\sigma_{i-1} \cdots \sigma_5)(\sigma_i \cdots \sigma_5).
\]
Let $b_i$ be the curve such that
\[
H_{b_i} = \rho(g_i)\sigma_5\rho(g_i)^{-1}. 
\]
Since $g_i$  commutes with $\sigma_1$ and $\sigma_3$, and hence with $\sigma_1\sigma_3^{-1}$, it follows that $\rho(g_i)$ commutes with $\rho(\sigma_1\sigma_3^{-1}) = \sigma_1\sigma_3^{-1}$. It follows further that $\rho(g_i)$ commutes with each of $\sigma_1$ and $\sigma_3$ (it cannot be that $\rho(g_i)$ interchanges $a_1$ and $a_3$ because the signs of the half-twists differ).  Hence $H_{b_i}$ commutes with $\sigma_1$. 
 
Using the above properties of $g_i$ and the fact (from Step 1) that $\rho(\sigma_1\sigma_5^{-1}) = \sigma_1\sigma_5^{-1}$, we have
 \begin{align*}
 \rho(\sigma_1\sigma_i^{-1}) =\rho(g_i(\sigma_1\sigma_5^{-1})g_i^{-1}) &= \rho(g_i)\rho(\sigma_1\sigma_5^{-1})\rho(g_i)^{-1}
  = \rho(g_i)\sigma_1\sigma_5^{-1}\rho(g_i)^{-1}
  = \sigma_1H_{b_i}^{-1}.
 \end{align*}
This completes the second step.

\bigskip

\noindent\emph{Step 3. For $i\in \{2,4\}$ there exists a curve $b_i$ such that $\rho(\sigma_1\sigma_i^{-1}) = \sigma_1H_{b_i}^{-1}$.}  

\bigskip

First we treat the case $i=4$. Choose $g_4\in B_n'$ such that $g_4\sigma_3g_4^{-1} = \sigma_4$ and such that $g_4$ commutes with $\sigma_1$ and each $\sigma_i$ with $i\geq 6$ (for instance $g_4 = \sigma_4^{-2}\sigma_3 \sigma_4$). The second condition implies that $g_4$ commutes with $\sigma_1\sigma_6^{-1}$ and hence that $\rho(g_4)$ commutes with $\rho(\sigma_1\sigma_6^{-1}) = \sigma_1H_{b_6}^{-1}$. Equivalently, $\rho(g_4)$ commutes with $\sigma_1$ and $H_{b_6}$. Define $b_4$ to be the curve such that
\[
H_{b_4} = \rho(g_4)\sigma_3\rho(g_4)^{-1}
\]
We then have that
\[
\rho(\sigma_1\sigma_4^{-1}) = \rho(g_4(\sigma_1\sigma_3^{-1})g_4^{-1}) = \rho(g_4)(\sigma_1\sigma_3^{-1})\rho(g_4)^{-1} =  \sigma_1H_{b_4}^{-1}.
\]
This completes the $i=4$ case.

\bigskip

We now address the $i=2$ case;  this is similar to the $i=4$ case, but more complicated.  Choose $g_2\in B_n'$ such that $g_2\sigma_1g_2^{-1} = \sigma_2$ and such that $g_2$ commutes with each $\sigma_i$ with $i\geq 4$. The second condition implies that $g_2$ commutes with $\sigma_5\sigma_6^{-1}$ and hence that $\rho(g_2)$ commutes with $\rho(\sigma_5\sigma_6^{-1})$.  The latter is equal to $\sigma_5H_{b_6}^{-1}$; indeed,
\[
\rho(\sigma_5\sigma_6^{-1}) = \rho((\sigma_1\sigma_5^{-1})^{-1}(\sigma_1\sigma_6^{-1})) = (\sigma_1\sigma_5^{-1})^{-1}\sigma_1H_{b_6}^{-1} = \sigma_5H_{b_6}^{-1}.
\]
Thus, $\rho(g_2)$ commutes with $\sigma_5H_{b_6}^{-1}$. 

Next, the element $g_2$ commutes with $\sigma_4\sigma_6^{-1}$ so $\rho(g_2)$ commutes with $\rho(\sigma_4\sigma_6^{-1})$.  We have
\[
\rho(\sigma_4\sigma_6^{-1}) = \rho((\sigma_1\sigma_4^{-1})^{-1}(\sigma_1\sigma_6^{-1})) = (\sigma_1H_{b_4}^{-1})^{-1}(\sigma_1H_{b_6}^{-1}) = H_{b_4}H_{b_6}^{-1}
\]
Thus $\rho(g_2)$ also commutes with $H_{b_4}H_{b_6}^{-1}$.  It follows that $\rho(g_2)$ commutes with $H_{b_6}$ (as in Step 2, we can conclude this because the signs on $H_{b_4}$ and $H_{b_6}^{-1}$ differ). 

By the previous two pargraphs, $\rho(g_2)$ commutes with $\sigma_5H_{b_6}^{-1}$ and $H_{b_6}$.  It follows that $\rho(g_2)$ commutes with $\sigma_5$.  

Define $b_2$ to be the curve such that
\[
H_{b_2} = \rho(g_2)\sigma_1\rho(g_2)^{-1}
\]
We then have that
\[
\rho(\sigma_2\sigma_5^{-1}) = \rho(g_2(\sigma_1\sigma_5^{-1})g_2^{-1}) = \rho(g_2)(\sigma_1\sigma_5^{-1})\rho(g_2)^{-1} = \rho(g_2)\sigma_1\rho(g_2)^{-1}\sigma_5^{-1} = H_{b_2}\sigma_5^{-1}.
\]
Now
\[
\rho(\sigma_1\sigma_2^{-1}) = \rho((\sigma_1\sigma_5^{-1})(\sigma_2\sigma_5^{-1})^{-1}) = (\sigma_1\sigma_5^{-1})(H_{b_2}\sigma_5^{-1})^{-1} = \sigma_1H_{b_2}^{-1},
\]
as desired.

\bigskip

\noindent \emph{Step 4. The curves $a_1,b_2,a_3,b_4,a_5,b_6,\dots$ form a chain.} 

\bigskip

For $i$ odd let $b_i$ be the standard curve $a_i$; so we need to show that $b_1,\dots,b_{n-1}$ form a chain.  It follows from the definition of the $b_i$ and the fact that each $\sigma_i\sigma_j^{-1}$ is conjugate in $\B_n'$ to $\sigma_1\sigma_3^{-1}$ that each $b_i$ surrounds exactly two marked points.  To complete this step we must show that $i(b_i, b_j) = 0$ if $j-i \geq 2$ and that $i(b_i,b_{i+1}) = 2$ for each $1 \leq i \leq n-2$. 

We begin by showing that $i(b_i, b_j) = 0$ if $j-i \geq 2$.  In this case we have
\[
\rho(\sigma_i\sigma_j^{-1}) = \rho((\sigma_1\sigma_i^{-1})^{-1}(\sigma_1\sigma_j^{-1}))=H_{b_i}H_{b_j}^{-1}.
\]
Since $\sigma_i\sigma_j^{-1}$ is conjugate in $\B_n'$ to $\sigma_1\sigma_3^{-1}$ and since $\rho$ fixes the latter, it follows that $H_{b_i}H_{b_j}^{-1}$ is conjugate to $\sigma_1\sigma_3^{-1}$.  From this it follows that $b_i$ and $b_j$ are disjoint, as desired (if $b_i$ and $b_j$ were not disjoint, then $H_{b_i}H_{b_j}^{-1}$ would be a partial pseudo-Anosov braid).  

We now proceed to show that $i(b_i,b_{i+1}) = 2$ for each $1 \leq i \leq n-2$.  Here, it suffices to show that $H_{b_i}$ and $H_{b_{i+1}}$ satisfy the braid relation.  We already showed in Step 1 that $\sigma_1$ satisfies a braid relation with $H_{b_2}$ and so the $i=1$ case is settled. It remains to treat the cases $i \geq 3$ and $i=2$.  

First, fix some $i \geq 3$.  Since $\sigma_i$ satisfies a braid relation with $\sigma_{i+1}$, the element $\sigma_1\sigma_i^{-1}$ satisfies a braid relation with $\sigma_1\sigma_{i+1}^{-1}$. It follows that $\rho(\sigma_1\sigma_i^{-1})=\sigma_1H_{b_i}^{-1}$ satisfies a braid relation with $\rho(\sigma_1\sigma_{i+1}^{-1})=\sigma_1H_{b_{i+1}}^{-1}$. Since both $H_{b_i}$ and $H_{b_{i+1}}$ commute with $\sigma_1$, this implies that $H_{b_i}$ satisfies a braid relation with $H_{b_{i+1}}$, as desired.

Finally, we show that $H_{b_2}$ and $H_{b_3}$ satisfy the braid relation.  Similar to the previous paragraph, $\rho(\sigma_2\sigma_5^{-1})$ satisfies a braid relation with $\rho(\sigma_3\sigma_5^{-1})$. In Step 3 we showed that $\rho(\sigma_2\sigma_5^{-1}) = H_{b_2}\sigma_5^{-1}$.  Since $n \geq 6$ we also have
\[
\rho(\sigma_3\sigma_5^{-1}) = \rho\left((\sigma_1\sigma_3^{-1})^{-1}(\sigma_1\sigma_5^{-1})\right) = (\sigma_1\sigma_3^{-1})^{-1}(\sigma_1\sigma_5^{-1}) = \sigma_3\sigma_5^{-1}.
\]
It follows that $H_{b_2}$ and $\sigma_3=H_{b_3}$ satisfy the braid relation.  The completes the fourth step.

\bigskip

\noindent \emph{Step 5. The homomorphism $\rho$ is equivalent to the inclusion map.} 

\bigskip

Since the curves $b_1,b_2,\dots,b_{n-1}$ from Step 4 form a chain, there is an element $\alpha$ of $\B_n$ such that the curves $\alpha(b_1),\ldots, \alpha(b_{n-1})$ are equal to $a_1,\dots,a_{n-1}$, respectively (this is an instance of the change of coordinates principle for mapping class groups~\cite[Section 1.3]{primer}).

After replacing $\rho$ by its post-composition with the inner automorphism of $\B_n$ induced by $\alpha$, we have that 
\[
\rho(\sigma_1\sigma_i^{-1}) = \sigma_1\sigma_i^{-1}
\]
Since the elements $\sigma_1\sigma_i^{-1}$ generate $\B_n'$, it follows that $\rho$ is equal to the standard inclusion. This completes Step 5, and the theorem is proven. 
\end{proof}

%%%
%%%
%%%

\section{Homomorphisms between braid groups}
\label{sec:castel}

In this section we classify homomorphisms $\B_n \to \B_n$ for $n \geq 7$.  As discussed in the introduction, this is a special case of a theorem of Castel.  

Let $\rho : \B_n\rightarrow \B_n$ be a homomorphism, and let $k$ be an integer.  The \emph{transvection} of $\rho$ by $z^k$ is the homomorphism given by
\[
\rho^{z^k}(\sigma_i) = \rho(\sigma_i)z^k
\]
for all $1 \leq i \leq n-1$.  There is an equivalence relation on the set of homomorphisms $\B_n \to \B_n$ whereby $\rho_1 \sim \rho_2$ if $\rho_2 = \alpha \circ \rho_1^{z^k}$ for some automorphism $\alpha$ of $\B_n$ and some $k \in \Z$.  This notion of equivalence is more complicated than the one we defined for homomorphisms $\B_n' \to \B_n$ in the introduction, in that it involves the transvections.  There are no analogous transvections of homomorphisms $\B_n' \to \B_n$, since the image of $\B_n'$ must lie in $\B_n'$, and the only power of $z^k$ in $\B_n'$ is the identity.  

The following theorem represents the special case of Castel's theorem that we will prove.

\begin{theorem}[Castel]
Let $n \geq 7$, and let $\rho : \B_n \to \B_n$ be a homomorphism whose image is not cyclic.  Then $\rho$ is equivalent to the identity.
\end{theorem}

\begin{proof}

Assume that $\rho: \B_n\rightarrow \B_n$ is a homomorphism with non-cyclic image. This is equivalent to the assumption that restriction $\rho'$ of $\rho$ to $\B_n'$ is nontrivial. Theorem~\ref{thm:main} then implies that there is an automorphism $\alpha$ of $\B_n$ such that $\alpha\circ \rho'$ is the identity. Thus replacing $\rho$ by $\alpha\circ \rho$, we may assume that $\rho'$ is the inclusion map.

We claim that $\rho(z)$ is equal to $z^k$ for some integer $k$. As in Section~\ref{sec:proof} let $p = n(n-1)$.  Since $z\in \B_n$ is central, we have that $\rho(z)$ commutes with each $\rho(\sigma_i^pz^{-1})$. Since 
\[
\rho(\sigma_i^pz^{-1}) = \rho'(\sigma_i^pz^{-1}) = \sigma_i^pz^{-1},
\]
it follows that $\rho(z)$ commutes with each $\sigma_i^p$, hence with each $\sigma_i$.  The claim follows.  

We next claim that $\rho(\sigma_i)^p = \sigma_i^pz^{k-1}$ for each $i$.  By the previous claim we indeed have
\[
\sigma_i^pz^{-1} = \rho(\sigma_i^pz^{-1}) = \rho(\sigma_i^p)\rho(z)^{-1} = \rho(\sigma_i^p)z^{-k} = \rho(\sigma_i)^pz^{-k},
\]
whence the claim.

We now claim that there exist integers $r$ and $s$ such that for all $1\leq i\leq n-1$ we have
\[
\rho(\sigma_i) = \sigma_i^rz^s.
\]
By the previous claim, the canonical reduction system of $\rho(\sigma_i)$ is equal to $a_i$. Since $\sigma_1$ commutes with $\sigma_j$ for $j \geq 3$, we have that $\rho(\sigma_1)$ fixes each $a_j$ with $j \geq 3$.  Thus the Nielsen--Thurston component of $\rho(\sigma_1)$ corresponding to the region between $a_1$ and the boundary of $\D_n$ cannot be pseudo-Anosov or a nontrival periodic element.  It follows that $\rho(\sigma_1) = \sigma_1^rz^s$ for some $r$ and $s$.  Since $\rho(\sigma_i)$ is conjugate to $\rho(\sigma_1)$ for $1 \leq i \leq n-1$, the claim follows.

Next we claim that $r=1$.  We have
\[
\sigma_1\sigma_3^{-1} = \rho(\sigma_1\sigma_3^{-1}) = \sigma_1^rz^s\sigma_3^{-r}z^{-s} = \sigma_1^r\sigma_3^{-r},
\]
whence the claim.

We now have that $\rho(\sigma_i) = \sigma_iz^s$ for $1 \leq i \leq n-1$.  The transvection of  $\rho$ by  $z^{-s}$ is then equal to the identity. This completes the proof of the theorem. 
\end{proof}

\bibliographystyle{plain}
\bibliography{bnprime}

\end{document}